\documentclass[11pt]{amsart}
\usepackage{amscd}
\usepackage{amsfonts,amssymb,amsmath,amsthm}
\usepackage{url}
\usepackage{color}
\usepackage{enumerate}
\usepackage{bm}
\usepackage{geometry}
\usepackage[all]{xy}

\newtheorem{thm}{Theorem}[section]
\newtheorem{cor}[thm]{Corollary}
\newtheorem{lem}[thm]{Lemma}

\theoremstyle{definition}
\newtheorem{defin}[thm]{Definition}
\newtheorem{rem}[thm]{Remark}
\newtheorem{exa}[thm]{Example}

\numberwithin{equation}{section}

\makeatletter
\@namedef{subjclassname@2020}{%
  \textup{2020} Mathematics Subject Classification}
\makeatother

\author[A. Itaba]{Ayako Itaba}
\address{
Department of Mathematics, 
Faculty of Science, Tokyo University of Science\\
1-3 Kagurazaka, Shinjyuku-ku, Tokyo, 162-8601, Japan}
\email{itaba@rs.tus.ac.jp}

\author[I. Mori]{Izuru Mori}
\address{
Department of Mathematics, 
Graduate School of Science, 
Shizuoka University\\
836 Ohya, Suruga-ku, Shizuoka, 422-8529, Japan}
\email{mori.izuru@shizuoka.ac.jp}


\newcommand{\Hom}{{\rm Hom}}
\newcommand{\Ext}{{\rm Ext}}

\newcommand{\bbP}{\mathbb{P}^2 }

\newcommand{\al}{\alpha}
\newcommand{\be}{\beta}
\newcommand{\ga}{\gamma}
\newcommand{\de}{\delta}

\newcommand{\ep}{\varepsilon}

\def\Spec{\mathsf{Spec}}
\def\Proj{\mathsf{Proj}}
\def\Specn{\mathsf{Spec}_{{\rm nc}}}
\def\Projn{\mathsf{Proj}_{{\rm nc}}}

\def\a{\alpha}
\def\b{\beta}
\def\c{\gamma}
\def\d{\delta}
\def\s{\sigma}
\def\t{\tau}
\def\l{\lambda}

\def\cA{\mathcal A}
\def\cO{\mathcal O}
\def\cV{\mathcal V}

\def\NN{\mathbb N}
\def\PP{\mathbb P}
\def\ZZ{\mathbb Z}

\def\mod{\mathsf{mod}}
\def\grmod{\mathsf{grmod}}
\def\tors{\mathsf{tors}}
\def\tails{\mathsf{tails}}

\def\Aut{{\rm Aut}}
\def\PGL{{\rm PGL}}
\def\id{{\rm id}}

\def\<{\langle}
\def\>{\rangle}

\begin{document}

\title
{Quantum Projective Planes Finite over their Centers} 
\begin{abstract}
For a $3$-dimensional quantum polynomial algebra $A=\mathcal{A}(E,\sigma)$, 
Artin-Tate-Van den Bergh showed that $A$ is finite over its center if and only if $|\sigma|<\infty$.  Moreover, Artin showed that if $A$ is finite over its center and $E\neq \PP^2$, then $A$ has a fat point module, which plays an important role in noncommutative algebraic geometry, however the converse is not true in general.  In this paper, we will show that, if $E\neq \PP^2$, then $A$ has a fat point module if and only if the quantum projective plane $\Projn A$ is finite over its center in the sense of this paper if and only if $|\nu^*\s^3|<\infty$ where $\nu$ is the Nakayama automorphism of $A$.  In particular, we will show that if the second Hessian of $E$ is zero, then $A$ has no fat point module. 
\end{abstract}
\maketitle

\section{Introduction}

A quantum polynomial algebra is a noncommutative analogue of a commutative polynomial algebra, and a quantum projective space is the noncommutative projective scheme associated to a quantum polynomial algebra, so they are the most basic objects to study in noncommutative algebraic geometry.  In fact, the starting point of the subject noncommutative algebraic geometry is the paper \cite{ATV1} by Artin-Tate-Van den Bergh, showing that there exists a nice correspondence between $3$-dimensional quantum polynomial algebras $A$ and geometric pairs $(E, \s)$ where $E=\PP^2$ or a cubic divisor in $\PP^2$, and $\s\in \Aut E$, so the classification of $3$-dimensional quantum polynomial algebras reduces to the classification of ``regular'' geometric pairs.  Write $A=\cA(E, \s)$ for a $3$-dimensional quantum polynomial algebra corresponding to the geometric pair $(E, \s)$.  The geometric property of the geometric pair $(E, \s)$ provides some algebraic property of $A=\cA(E, \s)$.  One of the most striking result of such is in the companion paper \cite{ATV2}.

\begin{thm}
[{\cite[Theorem 7.1]{ATV2}}]  
Let $A=\cA(E, \s)$ be a $3$-dimensional quantum polynomial algebra.  Then 
$|\s|<\infty$ if and only if $A$ is finite over its center. 
\end{thm} 

Let $A=\cA(E, \s)$ be a $3$-dimensional quantum polynomial algebra.  To prove the above theorem, fat points of a quantum projective plane $\Projn A$ plays an essential role.  By Artin \cite {A}, 
if $A$ is finite over its center and $E\neq \mathbb{P}^{2}$, 
then $\Projn A$ has a fat point, however, the converse is not true.  To check the existence of a fat point, there is more important notion than $|\s|$, namely, 
$$\|\sigma \|:={\rm inf}\{i\in \mathbb{N}^{+} \mid 
 \sigma^{i}=\phi|_{E} \text{ for some } \phi\in {\rm Aut}\mathbb{P}^2\}.$$ 
In fact,  $\Projn A$ has a fat point if and only if $1<\|\s\|<\infty$ by \cite {A}.

In \cite{Mo2}, the notion that $\Projn A$ is finite over its center was introduced, and the following result was proved. 

\begin{thm}[{\cite[Theorem 4.17]{Mo2}}]
 Let $A=\cA(E, \s)$ be a $3$-dimensional quantum polynomial algebra such that $E\subset \PP^2$ is a triangle.  Then 
$\|\s\|<\infty$ if and only if $\Projn A$ is finite over its center. 
\end{thm} 

The purpose of this paper is to extend the above theorem to all $3$-dimensional quantum polynomial algebras.  In fact, the following is our main result.    

\begin{thm} [Theorem \ref{q.nu}, Corollary \ref{cor.main}]  
Let $A=\cA(E, \s)$ be a $3$-dimensional quantum polynomial algebra such that $E\neq \PP^2$, and $\nu\in \Aut A$  the Nakayama automorphism of $A$.  Then 
the following are equivalent:
\begin{enumerate}[{\rm (1)}]
\item{} $|\nu^*\s^3|<\infty$. 
\item{} $\|\s\|<\infty$. 
\item{} $\Projn A$ is finite over its center.
\item{} $\Projn A$ has a fat piont. 
\end{enumerate} 
\end{thm} 

As a biproduct, we have the following corollary.  

\begin{cor} Let $A=\cA(E, \s)$ be a $3$-dimensional quantum polynomial algebra. If the second Hessian of $E$ is zero, 
then $A$ is never finite over its center. 
\end{cor}

These results are important to study representation theory of the Beilinson algebra $\nabla A$, which is a typical example of a $2$-representation infinite algebra defined in \cite {HIO}.  This was the original motivation of the paper \cite {Mo2}.  
\section{Preliminaries}
Throughout this paper, we fix an algebraically closed field $k$ of characteristic $0$.  All algebras and (noncommutative) schemes are defined over $k$.  We further assume that all (graded) algebras 
are finitely generated (in degree $1$) over $k$, that is, algebras of the form $k\<x_1, \dots, x_n\>/I$ for some (homogeneous) ideal $I\lhd k\<x_1, \dots, x_n\>$ (where $\deg x_i=1$ for every $i=1, \dots, n$).   

\subsection{Geometric quantum polynomial algebras}

In this subsection, we define geometric algebras and quantum polynomial algebras.  

\begin{defin}[{\cite[Definition 4.3]{Mo1}}]
\label{geometric}
	A geometric pair $(E,\sigma)$ consists of a projective scheme 
        $E \subset \PP^{n-1}$
	and $\sigma \in {\rm Aut}_{k}\,E$.
	For a quadratic algebra $A=k\<x_1, \dots, x_n\>/I$ where $I\lhd k\<x_1, \dots, x_n\>$ is a homogeneous ideal generated by elements of degree 2, we define 
	$$
\mathcal{V}(I_2):=\{ (p,q)\in\mathbb{P}^{n-1}\times \PP^{n-1}
\,|\,f(p,q)=0\,\,{\rm for\,\,any\,\,} f \in I_2 \}.
$$
	\begin{enumerate}[{\rm (1)}]
		\item We say that {\it $A$ satisfies {\rm (G$1$)}} if there exists a geometric pair $(E,\sigma)$ such that
		$$
		\mathcal{V}(I_2)=\{ (p,\sigma(p)) \in \mathbb{P}^{n-1}\times \PP^{n-1}
		\,|\,p \in E \}.
		$$
		In this case, we write $\mathcal{P}(A)=(E,\sigma)$, 
                and call $E$ the {\it point scheme} of $A$. 
		\item We say that {\it $A$ satisfies {\rm (G$2$)}} if there exists a geometric pair $(E,\sigma)$ such that
		$$
		I_2=\{ f \in k\<x_1, \dots, x_n\>_2
		\,|\,f(p,\sigma(p))=0\,\,{\rm for\,\,any\,\,} p \in E \}.
		$$
		In this case, we write $A=\mathcal{A}(E,\sigma)$.
		\item A quadratic algebra $A$ is called {\it geometric} if $A$ satisfies both (G1) and (G2)
		with $A=\mathcal{A}(\mathcal{P}(A))$.
	\end{enumerate}	
\end{defin}

\begin{defin}  A right noetherian graded algebra $A$ is called a 
{\it $d$-dimensional quantum polynomial algebra} if 
\begin{enumerate}[{\rm (1)}]
\item{} $\operatorname{gldim} A=d$, 
\item{} $\Ext^i_A(k, A)\cong \begin{cases}  k & \textnormal { if } i=d, \\
0 & \textnormal { if }  i\neq d, \end{cases}$ 
and 
\item{} $H_A(t):=\sum_{i=0}^{\infty}(\dim _kA_i)t^i=(1-t)^{-d}$.
\end{enumerate}
\end{defin} 
 
Note that a $3$-dimensional quantum polynomial algebra 
is exactly the same as a $3$-dimensional quadratic AS-regular algebra, so we have the following result. 
\begin{thm}[\cite{ATV1}]
	\label{ATV1} Every $3$-dimensional quantum polynomial algebra is a geometric algebra where the point scheme is either $\PP^2$ or a cubic divisor in $\PP^2$. 
\end{thm}

\begin{rem} There exists a $4$-dimensional quantum polynomial algebra which is not a geometric algebra, however, as far as we know, there exists no example of a quantum polynomial algebra which does not satisfy (G1). 
\end{rem} 

We define a type of a $3$-dimensional quantum polynomial algebra $A=\cA(E, \s)$ in terms of the point scheme $E\subset \PP^2$.  
\begin{description}
\item[{\rm Type P}] 
      $E$ is $\mathbb{P}^{2}$. 
\item [{\rm Type S}] 
      $E$ is a triangle. 
\item[{\rm Type S'}] 
      $E$ is a union of a line and a conic meeting at two points. 
\item[{\rm Type T}]
      $E$ is a union of three lines meeting at one point. 
\item[{\rm Type T'}]
      $E$ is a union of a line and a conic meeting at one point. 
\item[{\rm Type NC}]
          $E$ is a nodal cubic curve. 
\item[{\rm Type CC}] $E$ is a cuspidal cubic curve. 
\item[{\rm Type TL}] $E$ is a triple line. 
\item[{\rm Type WL}]$E$ is a union of a double line and a line. 
\item[{\rm Type EC}] $E$ is an elliptic curve. 
\end{description}
\subsection{Quantum projective spaces finite over their centers}

\begin{defin} {\it A noncommutative scheme {\rm (}over $k${\rm )}} is a pair $X=(\mod X, \cO_X)$ consisting of a $k$-linear abelian category $\mod X$ and an object $\cO_X\in \mod X$.  We say that {\it two noncommutative schemes $X=(\mod X, \cO_X)$ and $Y=(\mod Y, \cO_Y)$ are isomorphic}, denoted by $X\cong Y$, if there exists an equivalence functor $F:\mod X\to \mod Y$ such that $F(\cO_X)\cong \cO_Y$. 
\end{defin} 

If $X$ is a commutative noetherian scheme, then we view $X$ as a noncommutative scheme by $(\mod X, \cO_X)$ where $\mod X$ is the category of coherent sheaves on $X$ and $\cO_X$ is the structure sheaf on $X$.  

Noncommutative affine and projective schemes are defined in \cite {AZ}. 

\begin{defin} 
If $R$ is a right noetherian algebra, then we define {\it the noncommutative affine scheme associated to $R$} by $\Specn R=(\mod R, R)$ where $\mod R$ is the category of finitely generated right $R$-modules and $R\in \mod R$ is the regular right module. 
\end{defin} 

Note that if $R$ is commutative, then $\Specn R\cong \Spec R$.  

\begin{defin}
If $A$ is a right noetherian graded algebra, $\grmod A$ is the category of finitely generated graded right $A$-modules, and $\tors A$ is the full subcategory of $\grmod A$ consisting of finite dimensional modules over $k$, then we define {\it the noncommutative projective scheme associated to $A$} by $\Projn A=(\tails A, \pi A)$ where $\tails A:=\grmod A/\tors A$ is the quotient category, $\pi:\grmod A\to \tails A$ is the quotient functor, and $A\in \grmod A$ is the regular graded right module. If $A$ is a $d$-dimensional quantum polynomial algebra, then we call $\Projn A$ {\it a quantum $\PP^{d-1}$}.  In particular,  if $d=3$, then we call $\Projn A$ {\it a quantum projective plane}.   
\end{defin} 

Note that if $A$ is commutative, then $\Projn A\cong \Proj A$.  It is known that if $A$ is a $2$-dimensional quantum polynomial algebra, then $\Projn A\cong \PP^1$.  


For a $3$-dimensional quantum polynomial algebra $A=\cA(E, \s)$, 
we have the following geometric characterization when $A$ 
is finite over its center. 
\begin{thm}[{\cite[Theorem 7.1]{ATV2}}]
\label{thm_{ATV2}}
Let $A=\mathcal{A}(E,\sigma)$ be 
a $3$-dimensional quantum polynomial algebra. 
Then the following are equivalent: 
\begin{enumerate}[{\rm (1)}]
\item $|\sigma|<\infty$. 
\item $A$ is finite over its center. 
\end{enumerate}
\end{thm}

Since the property that $A$ is finite over its center is not preserved under isomorphisms of noncommutative projective schemes 
$\Projn A$, we will make the following rather ad hoc definition.  

\begin{defin}
\label{def_{Z(proj)}}
Let $A$ be a $d$-dimensional quantum polynomial algebra. 
We say that $\mathsf{Proj}_{{\rm nc}}A$ is \textit{finite over its center} 
if there exists a $d$-dimensional quantum polynomial algebra $A'$ 
finite over its center such that 
$\Projn A\cong \Projn A'$.  
\end{defin}

For a $3$-dimensional quantum polynomial algebra, the above definition coincides with \cite[Definition 4.14]{Mo2} by the following result. 

\begin{lem}[{\cite[Corollary A.10]{AOU}}]
\label{lem.AOU}
Let $A$ and $A'$ be $3$-dimensional quantum polynomial algebras.  Then $\grmod A\cong \grmod A'$ if and only if $\Projn A\cong \Projn A'$. 
\end{lem} 

To characterize ``geometric'' quantum projective spaces finite over their centers, we will introduce the following notion: 

\begin{defin}[{\cite[Definition 4.6]{Mo2}}]
For a geometric pair $(E, \s)$ 
where $E\subset \mathbb{P}^{n-1}$ and $\sigma\in {\rm Aut}_{k}E$, 
we define 
$$\Aut_k(\PP^{n-1}, E):=\{\phi|_E\in \Aut_k E\mid \phi\in \Aut_k \PP^{n-1}\},$$
and 
$$
\|\sigma \|:={\rm inf}\{i\in \mathbb{N}^{+} \mid 
\sigma^{i}\in \Aut_k (\PP^{n-1}, E)\}.
$$
\end{defin}
For a geometric pair $(E, \s)$, clearly $\|\sigma\| \leq |\sigma|$.  The following are the basic properties of $\|\s\|$.  

\begin{lem}[{\cite [Lemma 2.5]{MU1}, \cite[Lemma 4.16 (1)]{Mo2}}]
\label{lem_Mo2}
Let $A$ and $A'$ be $d$-dimensional quantum polynomial algebras 
satisfying {\rm (G1)} with 
$\mathcal{P}(A)=(E,\sigma)$ and $\mathcal{P}(A')=(E',\sigma')$. 
\begin{enumerate}[{\rm (1)}]
\item{} If $A\cong A'$, then $E\cong E'$ and $|\s|=|\s'|$. 
\item{} If ${\rm grmod}\,A\cong {\rm grmod}\,A'$, then $E\cong E'$ and $||\s||=||\s'||$. 
\end{enumerate}
In particular, if $A$ and $A'$ are $3$-dimensional quantum polynomial algebras such that $\Projn A\cong \Projn A'$, then $E\cong E'$ 
{\rm (}that is, $A$ and $A'$ are of the same type{\rm )} and $||\s||=||\s'||$.  
\end{lem}

For a $3$-dimensional quantum polynomial algebra $A=\cA(E, \s)$ 
of Type S, 
we have the following geometric characterization when a quantum projective plane $\mathsf{Proj}_{{\rm nc}}A$ 
is finite over its center. 
\begin{thm}[{\cite[Theorem 4.17]{Mo2}}]
\label{thm_{Mo2}}
Let $A=\mathcal{A}(E,\sigma)$ be 
a $3$-dimensional quantum polynomial algebra of Type S. 
Then the following are equivalent{\rm :} 
\begin{enumerate}[{\rm (1)}]
\item 
$\|\sigma \|<\infty$. 
\item 
$\mathsf{Proj}_{{\rm nc}}A$ is finite over its center. 
\end{enumerate}
\end{thm}

The purpose of this paper is to extend the above theorem to all types.

\subsection{Points of a noncommutative scheme}

\begin{defin} Let $R$ be an algebra.  {\it A point of $\Specn R$} is an isomorphism class of a simple right $R$-module $M\in \mod R$ such that $\dim_kM<\infty$.  A point $M$ is called {\it fat} if $\dim _kM>1$. 
\end{defin}

\begin{rem} If $R$ is a commutative algebra and $p\in \Spec A$ is a closed point, then $A/\frak m_p\in \mod R$ is a point where $\frak m_p$ is the maximal ideal of $R$ corresponding to $p$.  In fact, this gives a bijection between the set of closed points of $\Spec R$ and the set of points of $\Specn R$.  In this commutative case, there exists no fat point.  
\end{rem} 

\begin{rem} Fat points are not preserved under Morita equivalences.  For example, $\mod k\cong \mod M_2(k)$, but it is easy to see that $\Specn k$ has no fat point while $\Specn M_2(k)$ has a fat point.  However, since $\Specn R\cong \Specn R'$ if and only if $R\cong R'$, fat points are preserved under isomorphisms of $\Specn R$. 
\end{rem} 

\begin{exa} \label{ex.wy} If $R=k\<u, v\>/(uv-vu-1)$ is the $1$st Weyl algebra, then it is well-known that there exists no finite dimensional right $R$-module, so $\Specn R$ has no point at all. 
\end{exa} 

\begin{exa}[{cf. \cite{S}}]
\label{ex.Lie}
If $R=k\<u, v\>/(vu-uv-u)$ is the enveloping algebra of a $2$-dimensional non-abelian Lie algebra, then the set of points of $\Specn R$ is given by $\{R/uR+(v-\mu)R\}_{\mu\in k}$, so $\Specn R$ has no fat point.  
In fact, the linear map $\d:k[u]\to k[u]$ defined by $\d(f(u))=uf'(u)$ is a derivation of $k[u]$ such that $R=k[u][v; \d]$ is the Ore extension so that $vf(u)=f(u)v+uf'(u)$. If $M$ is a finite dimensional right $R$-module, then there exists $f(u)=a_du^d+\cdots +a_1u+a_0\in k[u]\subset R$ of the minimal degree $\deg f(u)=d\geq 1$ such that $Mf(u)=0$.  Since $uf'(u)=vf(u)-f(u)v$, $M(df(u)-uf'(u))=0$ such that $\deg(df(u)-uf'(u))<\deg f(u)$, so $df(u)=uf'(u)$ by minimality of $\deg f(u)=d\geq 1$, but this is possible only if $f(u)=a_1u$, so $Mu=0$.  It follows that $M$ can be viewed as an $R/(u)$-module, a point of $\Specn (R/(u))\cong \Specn k[v]$, so $M\cong R/uR+(v-\mu)R$ for some $\mu\in k$.  Since $\Specn (R/(u))\cong \Specn k[v]$ is a commutative scheme, $\Specn R$ has no fat point. 
\end{exa} 

\begin{exa}[{\cite[Lemma 4.19]{Mo2}}]
\label{ex.q}
If $R=k\<u, v\>/(uv+vu)$ is a $2$-dimensional (ungraded) quamtum polynomial algebra, then the set of points of $\Specn R$ is given by 
$$\{R/(u-\l)R+vR\}_{\l\in k}\cup\{R/uR+(v-\mu)R\}_{\mu\in k}\cup\{R/(x^2-\l)R+(\sqrt{\mu}x+\sqrt{-\l}y)R+(y^2-\mu)R\}_{0\neq \l, \, \mu\in k}.$$
Among them, $\{R/(x^2-\l)R+(\sqrt{\mu}x+\sqrt{-\l}y)R+(y^2-\mu)R\}_{0\neq \l, \, \mu\in k}$ is the set of fat points of $\Specn R$.  
\end{exa}


\begin{defin} Let $A$ be a graded algebra.  
{\it A point of $\Projn A$} is an isomorphism class of a simple object of the form $\pi M\in \tails A$ where $M\in \grmod A$ is a graded right $A$-module such that $\lim_{i\to \infty}\dim _kM_i<\infty$.  A point $\pi M$ is called {\it fat} if $\lim _{i\to \infty}\dim _kM_i>1$, and, in this case, $M$ is called {\it a fat point module over $A$}.   
\end{defin} 

\begin{rem} If $A$ is a graded commutative algebra and $p\in \Proj A$ is a closed point, then $\pi (A/\frak m_p)\in \tails A$ is a point where $\frak m_p$ is the homogeneous maximal ideal of $A$ corresponding to $p$.  In fact, this gives a bijection between the set of closed points of $\Proj A$ and the set of points of $\Projn A$.  In this commutative case, there exists no fat point.  
\end{rem}

\begin{rem} It is unclear that fat points are preserved under isomorphisms of $\Projn A$ in general. However, fat point modules are preserved under graded Morita equivalences, so if $A$ and $A'$ are both $3$-dimensional quantum polynomial algebras such that $\Projn A\cong \Projn A'$, then there exists a natural bijection between the set of fat points of $\Projn A$ and that of $\Projn A'$ by Lemma \ref{lem.AOU}.
\end{rem} 

The following facts will be used to prove our main results.  

\begin{lem} [{\cite{Mo2}, \cite{A}}] Let $A=\cA(E, \s)$ be a $3$-dimensional quantum polynomial algebra.  
\label{lem.general}
\begin{enumerate}[{\rm (1)}]
\item{} 
$\|\sigma\|=1$ if and only if $E=\mathbb{P}^2$.  
\item{} 
$1<\|\s\|<\infty$ if and only if $\Projn A$ has a fat point.  
\end{enumerate}
\end{lem}

\begin{thm}
[{\cite[Theorem 4.20]{Mo2}}] 
\label{thm.dec}
If $A$ is a quantum polynomial algebra and $x\in A$ is a homogeneous normal element of positive degree, then there exists a bijection between the set of points of $\Projn A$ and the disjoint union of the set of points of $\Projn A/(x)$ and the set of points of $\Specn A[x^{-1}]_0$.  In this bijection, fat points correspond to fat points.   
\end{thm}

\section{Main results}
In this section, we will state and prove our main results. 

Let $A$ be a graded algebra and $\nu\in \Aut A$ a graded algebra automorphism.  For a graded $A$-$A$-bimodule $M$, we define a new graded $A$-$A$ bimodule $M_{\nu}=M$ as a graded vector space with the new actions $a*m*b:=am\nu(b)$ for $a, b\in A, m\in M$. Let $A$ be a $d$-dimensional quantum polynomial algebra.  {\it The canonical module of $A$} is defined by 
$$\omega _A:=\lim_{i\to \infty}\Ext^d_A(A/A_{\geq i}, A),$$
which has a natural graded $A$-$A$ bimodule structure.  It is known that there exists $\nu\in \Aut A$ such that $\omega_A\cong A_{\nu^{-1}}(-d)$ as graded $A$-$A$ bimodules.  We call $\nu$ 
{\it the Nakayama automorphism of $A$}.  
Among quantum polynomial algebras, Calabi-Yau quantum polynomial algebras defined below are easier to handle.

\begin{defin} A quantum polynomial algebra $A$ is called {\it Calabi-Yau} if the Nakayama automorphism of $A$ is the identity.
\end{defin}   

The following theorem plays an essential role to prove our main results, claiming that every quantum projective plane has a $3$-dimensional Calabi-Yau quantum polynomial algebra as a homogeneous coordinate ring. 

\begin{thm}[{\cite[Theorem 4.4]{IM2}}]
\label{Main2}
For every 
$3$-dimensional quantum polynomial algebra $A$,
there exists a $3$-dimensional Calabi-Yau quantum polynomial algebra 
$A'$
such that $\grmod A\cong \grmod A'$
so that $\Projn A\cong \Projn A'$. 
\end{thm}

By the above theorem, the proofs of our main results reduce to the Calabi-Yau case.  

\subsection{Calabi-Yau case} 

\noindent
Let $E=\cV(x^3+y^3+z^3-\l xyz)\subset \PP^2, \; \l\in k, \l^3\neq 27$ be an elliptic curve in the Hesse form. 
We fix a group structure with the identity element $o:=(1,-1,0)\in E$, and write 
$E[n]:=\{p\in E\mid np=o\}$ {\it the set of $n$-torsion points}. We also denote by $\s_p\in\Aut_k E$ {\it the translation automorphism by a point $p\in E$}.  It is known that $\s_p\in \Aut_k(\PP^2, E)$ if and only if $p\in E[3]$ (cf. \cite[Lemma 5.3]{Mo1}).  



\begin{lem}
\label{SP}
Denote a $3$-dimensional Calabi-Yau quantum polynomial algebra
as 
$$A=k\langle x,y,z \rangle/(f_{1}, f_{2}, f_{3})=\mathcal{A}(E,\sigma).$$
Then Table $1$ below gives a list of defining relations 
$f_{1}, f_{2}, f_{3}$ and the corresponding geometric pairs 
$(E,\sigma)$ for such algebras up to isomorphism. In Table $1$, we remark that: 
\begin{enumerate}[{\rm (1)}]
\item{} Type S and Type T are further divided into Type S$_1$, Type S$_3$ and Type T$_1$, Type T$_3$, respectively, in term of the form of $\s$.
\item{} The point scheme $E$ may consists of several irreducible components, and, in this case, $\sigma$ is described on each component.   
\item{} For Type NC and Type CC, $\s$ in Table 1 is defined except for the unique singular point $(0, 0, 1)\in E$, which is preserved by $\s$.
\item{} For Type TL and Type WL, $E$ is non-reduced, and description of $\s$ is omitted.
\end{enumerate} 
\end{lem}
\begin{center}
{\renewcommand\arraystretch{1.1}
{\small
\begin{tabular}{|c|p{5.3cm}|p{2.8cm}|p{6.0cm}|}
\multicolumn{4}{c}{Table $1$}
\\[5pt]
\hline
           Type
           & $f_{1}, f_{2}, f_{3}$
           & $E$
           & $\sigma$
            \\ \hline\hline
$\rm{P}$ 
            & 
           $
           \left\{
             \begin{array}{ll}
             yz-\alpha zy\\
             zx-\alpha xz\quad\quad\alpha^{3}=1\\
             xy-\alpha yx
             \end{array}
             \right.
             $
           & $\mathbb{P}^{2}$
           &  \rule{0pt}{25pt}
           $\sigma(a,b,c)=(a,\alpha b,\alpha^{2}c)$
           \\[18pt] \hline
$\rm{S}_1$ 
           & 
           $\left\{
             \begin{array}{ll}
              yz-\alpha zy\\
              zx-\alpha xz \quad\quad \alpha^{3}\neq 0,1\\
              xy-\alpha yx
             \end{array}
             \right.
             $
           & $\begin{array}{ll}
           & \cV(x) \\
           \cup & \cV(y) \\
           \cup & \cV(z)
           \end{array}$
           & \rule{0pt}{25pt}
           $\left\{
             \begin{array}{ll}
              \sigma(0,b,c)=(0,b,\alpha c)\\
              \sigma(a,0,c)=(\alpha a,0,c)\\
               \sigma(a,b,0)=(a,\alpha b,0)
             \end{array}
             \right.
             $
           \\[18pt] \hline
$\rm{S}_3$ 
          & 
           $\left\{
             \begin{array}{ll}
             zy-\alpha x^{2}\\
             xz-\alpha y^{2}\quad\quad \alpha^{3}\neq 0,1\\
             yx-\alpha z^{2}
             \end{array}
             \right.
             $
           & $\begin{array}{ll}
           & \cV(x) \\
           \cup & \cV(y) \\
           \cup & \cV(z)
           \end{array}$
           & \rule{0pt}{25pt}
           $\left\{
             \begin{array}{ll}
              \sigma(0,b,c)=(\alpha c, 0,b)\\
              \sigma(a,0,c)=(c, \alpha a,0)\\
              \sigma(a,b,0)=(0,a,\alpha b)
             \end{array}
             \right.
             $
           \\[18pt] \hline
\rm{S'}
            &
            $\left\{
             \begin{array}{ll}
               yz-\alpha zy+x^{2}\\
               zx-\alpha xz\quad\quad \alpha^{3}\neq 0,1\\
               xy-\alpha yx
             \end{array}
             \right.
             $
            & $\begin{array}{ll}
           & \cV(x) \\
           \cup & \cV(x^2-\l yz) \\
           & \l=\frac{\a^3-1}{\a}
           \end{array}$
            & \rule{0pt}{25pt}
            $\left\{
             \begin{array}{ll}
             \sigma(0,b,c)=(0,b, \alpha c)\\
             \sigma(a,b,c)=(a,\alpha b,\alpha^{-1} c)\\
             \end{array}
             \right.
             $
            \\[18pt] \hline
$\rm{T_1}$ 
          & 
            $\left\{
             \begin{array}{ll}
             yz-zy+xy+yx-y^{2}\\
             zx-xz+x^{2}-yx-xy\\
             xy-yx
             \end{array}
             \right.
             $
            & $\begin{array}{ll}
           & \cV(x) \\
           \cup & \cV(y) \\
           \cup & \cV(x-y)
           \end{array}$
            & 
            $\left\{
             \begin{array}{ll}
             \sigma(0,b,c)=(0,b,b+c)\\
             \sigma(a,0,c)=(a,0,a+c)\\
             \sigma(a,a,c)=(a,a,-a+c)
             \end{array}
             \right.
             $
           \\ \hline
%
$\rm{T_3}$ 
           & 
            $\left\{
             \begin{array}{ll}
             yz-xy-yx+y^{2}-xz-zx \\
             \hfill +x^{2}\\
             zx-x^{2}+xy+yx-zy-yz \\ 
             \hfill -y^{2}\\
             xy-x^{2}-y^{2}
             \end{array}
             \right.
             $
           & $\begin{array}{ll}
           & \cV(x) \\
           \cup & \cV(y) \\
           \cup & \cV(x-y)
           \end{array}$
            & \rule{0pt}{35pt}
            $\left\{
             \begin{array}{ll}
             \sigma(0,b,c)=(b,0, b+c) \\
             \sigma(a,0,c)=(a,a,-c) \\
             \sigma(a,a,c)=(0,a,-c)
             \end{array}
             \right.
             $
            \\[30pt] \hline
\rm{T'} 
          & 
            $\left\{
             \begin{array}{ll}
            yz-zy+xy+yx\\
            zx-xz+x^{2}-yz-zy+y^{2}\\
            xy-yx-y^{2}
             \end{array}
             \right.
             $
            & $\begin{array}{ll}
           & \cV(y) \\
           \cup & \cV(x^2-yz)
           \end{array}$
            & 
            $\left\{
             \begin{array}{ll}
             \sigma(a,0,c)=(a,0,a+c), \\
             \sigma(a,b,c)=(a-b,b,-2a+b+c) 
             \end{array}
             \right.
             $
            \\ \hline
$\rm{NC}$ 
           & 
            $\left\{
             \begin{array}{ll}
               yz-\alpha zy+x^{2}\\
               zx-\alpha xz+y^{2}\quad\quad \alpha^{3}\neq 0,1\\
               xy-\alpha yx
             \end{array}
             \right.
             $
            & $\begin{array}{ll}
            \mathcal{V}(x^{3}+y^{3} \\
            \qquad -\l xyz) \\
            \l=\frac{\a^3-1}{\a}\end{array}$
            & \rule{0pt}{25pt}
             $
             \begin{array}{ll}
             \sigma(a,b,c)=\\
             (a,\alpha b,
             -\frac{a^{2}}{b}+\alpha^{2}c) \\
             \end{array}
             $
           \\[18pt] \hline
$\rm{CC}$ 
         & 
         $\left\{
             \begin{array}{ll}
            yz-zy+y^{2}+3x^{2}\\
            zx-xz+yx+xy-yz-zy\\
            xy-yx-y^{2}
             \end{array}
             \right.
             $
          & $\mathcal{V}(x^{3}-y^{2}z)$
          & $
             \begin{array}{ll}
               \sigma(a,b,c)=\\
               (a-b,b,
               -3\frac{a^2}{b}+3a-b+c) 
             \end{array}
             $
          \\ \hline
$\rm{TL}$
           & 
             $\left\{
             \begin{array}{ll}
             yz-\alpha zy -x^{2}\\
             zx-\alpha xz \quad\quad \alpha^{3}=1\\
             xy-\alpha yx
             \end{array}
             \right.
             $
           & $ \mathcal{V}(x^{3})$
            & \rule{0pt}{25pt}
             omitted
             \\[18pt] \hline
$\rm{WL}$ 
             & 
            $\left\{
             \begin{array}{ll}
             yz-zy-\dfrac{1}{3}y^{2}
             \vspace{0.5em} \\
             zx-xz-\dfrac{1}{3}(yx+xy)\\
             xy-yx
             \end{array}
             \right.
             $
            & $\mathcal{V}(x^{2}y)$
            & 
            \rule{0pt}{35pt}
             omitted 
            \\[25pt] \hline
EC 
             & 
            $\left\{
             \begin{array}{ll}
              \a yz+\b zy+\c x^{2} \\
              \a zx+\b xz+\c y^{2} \\
              \a xy+\b yx+\c z^{2}
             \end{array}
             \right.
             $
             
             where $p=(\a,\b,\c)\in E\backslash E[3]$
            & $\begin{array}{ll}
            \mathcal{V}(x^{3}+y^{3}+z^{3} \\
            \qquad -\lambda xyz), \\  
            \lambda=\frac{\a^3+\b^3+\c^3}{\a\b\c}
            \end{array}$
            & 
            \rule{0pt}{35pt}
              $\s_p$ where $p=(\a,\b,\c)\in E\backslash E[3]$
            \\[30pt] \hline
\end{tabular}
}}
\end{center}

\begin{proof} The list of the defining relations $f_1, f_2, f_3$ is given in \cite[Theorem 3.3]{IM1} and \cite[Corollary 4.3]{Ma}.  It is not difficult to calculate their corresponding geometric pairs $(E, \s)$ using the condition (G1) 
(see, for example, \cite[the proof of Theorem 3.1]{U} for Type P, S$_{1}$, S$_{3}$, S', and \cite[the proof of Theorem 3.6]{MU1} for Type T$_{1}$, T'.)
We only give some calculations to check that $(E, \s)$ in Table 1 is correct for Type CC. 

Let 
$A=k\langle x,y,z \rangle/(f_{1}, f_{2}, f_{3})$ be a 3-dimensional Calabi-Yau quantum polynomial algebra of Type CC
where 
$$f_1=yz-zy+y^{2}+3x^{2}, \;  
            f_2=zx-xz+yx+xy-yz-zy, \; 
            f_3=xy-yx-y^{2},$$
            and let $E=\mathcal{V}(x^{3}-y^{2}z)$, 
          and 
          $$\sigma(a,b,c)=\begin{cases} 
               (a-b,b,
               -3\dfrac{a^2}{b}+3a-b+c) & \textnormal { if } (a, b, c)\neq (0, 0, 1), \\
               (0, 0, 1) & \textnormal { if } (a, b, c)=(0, 0, 1) 
             \end{cases}$$
as in Table 1.  If $p=(a, b, c)\in E$, then $a^3-b^2c=0$, 
so  
\begin{align*}
f_1(p, \s(p)) &= f_1((a, b, c), (a-b,b,
               -3\dfrac{a^2}{b}+3a-b+c)) \\
               &= \;  b(-3\dfrac{a^2}{b}+3a-b+c)-cb+b^2+3a(a-b) \\
               &= \;  -3a^2+3ab-b^2+bc-bc+b^2+3a^2-3ab  =0, \\
f_2(p, \s(p)) &= f_2((a, b, c), (a-b,b,
               -3\dfrac{a^2}{b}+3a-b+c))\\
               &= \;  c(a-b)-a(-3\dfrac{a^2}{b}+3a-b+c)+b(a-b)+ab-b(-3\dfrac{a^2}{b}+3a-b+c)-cb \\
               &= \;  ac-bc+3\dfrac{a^3}{b}-3a^2+ab-ac+ab-b^2+ab+3a^2-3ab+b^2-bc-bc \\
               &= \;  \dfrac{3}{b}(a^3-b^2c)=0, \\
f_3(p, \s(p)) &= f_3((a, b, c), (a-b,b,
               -3\dfrac{a^2}{b}+3a-b+c)) \\
               &= \;  ab-b(a-b)-b^2 = ab-ab+b^2-b^2 = 0, 
\end{align*}
hence $\{(p, \s(p))\in \PP^2\times \PP^2\mid p\in E\}\subset \cV(f_1, f_2, f_3)$.  Since $E\subset \PP^2$ is a cuspidal cubic curve (and we know that the point scheme of $A$ is not $\PP^2$), $E$ is the point scheme of $A$, so ${\mathcal P}(A)=(E, \s)$. 
\end{proof}

\begin{thm} \label{q.main}
If $A=\cA(E, \s)$ is a $3$-dimensional Calabi-Yau quantum polynomial algebra, then $||\s||=|\s^3|$, so the following are equivalent:
\begin{enumerate}[{\rm (1)}]
\item{} $|\s|<\infty$. 
\item{} $||\s||<\infty$. 
\item{} $A$ is finite over its center. 
\item{} $\Projn A$ is finite over its center. 
\end{enumerate}
\end{thm} 
\begin{proof}
First, we will show that $||\s||=|\s^3|$ for each type using the defining relations $f_1, f_2, f_3$ and geometric pairs $(E, \s)$ given in Lemma \ref{SP}.  Recall that $\s^i\in \Aut_k(\PP^2, E)$ if and only if it is represented by a matrix in $\PGL_3(k)\cong \Aut_k \PP^2$.

\noindent
\underline{Type P}\quad
Since $\s^3=\id$, $||\s||=1=|\s^3|$. 

\medskip

\noindent
\underline{Type S$_{1}$}\quad
Since
$$
\begin{cases} 
\sigma^i(0,b,c)=(0,b,\alpha^i c), \\
\sigma^i(a,0,c)=(\alpha^i a,0,c)=(\a^{2i}a, 0, \a^ic), \\
\sigma^i(a,b,0)=(a,\alpha^i b,0)=(\a^{2i}a, \a^{3i}b, 0)
\end{cases}
$$
$\s^i\in \Aut_k(\PP^2, E)$ if and only if $\a^{3i}=1$, so $||\s||=|\a^3|=|\s^3|$.  

\medskip

\noindent
\underline{Type S$_{3}$}\quad
Since
$$
\begin{cases} 
\sigma^i(0,b,c)=(0,b,\alpha^i c)\begin{pmatrix} 0 & 1 & 0 \\ 0 & 0 & 1 \\ 1 & 0 & 0 \end{pmatrix}^i, \\
\sigma^i(a,0,c)=(\alpha^i a,0,c)\begin{pmatrix} 0 & 1 & 0 \\ 0 & 0 & 1 \\ 1 & 0 & 0 \end{pmatrix}^i, \\
\sigma^i(a,b,0)=(a,\alpha^i b,0)\begin{pmatrix} 0 & 1 & 0 \\ 0 & 0 & 1 \\ 1 & 0 & 0 \end{pmatrix}^i, 
\end{cases}
$$
and $\begin{pmatrix} 0 & 1 & 0 \\ 0 & 0 & 1 \\ 1 & 0 & 0 \end{pmatrix}\in \Aut_k (\PP^2, E)$, 
$\s^i\in \Aut_k(\PP^2, E)$ if and only if $\a^{3i}=1$, so $||\s||=|\a^3|=|\s^3|$.  

\medskip

\noindent
\underline{Type S'}\quad
Since 
$$
\begin{cases}
\sigma^i(0,b,c)=(0,b, \alpha^i c), \\
\sigma^i(a,b,c)=(a,\alpha^i b,\alpha^{-i} c)=(\a^{-i}a, b, \a^{-2i}c), \\
\end{cases}
$$
$\s^i\in \Aut_k(\PP^2, E)$ if and only if $\a^{3i}=1$, so $||\s||=|\a^3|=|\s^3|$.  

\medskip

\noindent
\underline{Type T$_{1}$}\quad
Since
$$
\begin{cases}
\sigma^{i}(0,b,c)=(0,b,ib+c), \\
\sigma^{i}(a,0,c)=(a,0,ia+c), \\
\sigma^{i}(a,a,c)=(a,a,-ia+c), 
\end{cases}
$$
$\s^i\not \in \Aut_k (\PP^2, E)$ for every $i\geq 1$, so $||\s||=\infty=|\s^3|$.  

\medskip

\noindent
\underline{Type T$_{3}$}\quad
Since
$$
\begin{cases}
\sigma^{3i}(0,b,c)=(0,b,ib+c), \\
\sigma^{3i}(a,0,c)=(a,0, ia+c), \\
\sigma^{3i}(a,a,c)=(a,a,-ia+c), 
\end{cases}
$$
$\s^{3i}\not \in \Aut_k (\PP^2, E)$ for every $i\geq 1$, so $||\s||=\infty=|\s^3|$.  

\medskip

\noindent
\underline{Type T'}\quad
Since
$$
\begin{cases}
\sigma^{i}(a,0,c)=(a,0,ia+c), \\
\sigma^{i}(a,b,c)=(a-ib,b,-2ia+i^2b+c), 
\end{cases}
$$
$\s^i\not \in \Aut_k (\PP^2, E)$ for every $i\geq 1$, so $||\s||=\infty=|\s^3|$.  

\medskip

\noindent
\underline{Type NC}\quad
Since 
$$\s^i(a, b, c)=(a, \a^i b, -\frac{\a^{3i}-1}{\a^{i-1}(\a^3-1)}\frac{a^2}{b}+\a^{2i} c),$$
$\s^i\in \Aut_k(\PP^2, E)$ if and only if $\a^{3i}=1$, so $||\s||=|\a^3|=|\s^3|$. 


\medskip

\noindent
\underline{Type CC}\quad
Since 
$$\s^i(a, b, c)=(a-ib, b, -3i\frac{a^2}{b}+3i^2a-i^3b+c),$$
$\s^{i}\not \in \Aut (\PP^2, E)$ for every $i\geq 1$, so $||\s||=\infty=|\s^3|$.  

\medskip

\noindent
\underline{Type TL}\quad
Since $A=k\<x, y, z\>/(yz-\a zy-x^2, zx-\a xz, xy-\a yx), \; \a^3=1$, we see that $x\in A_1$ is a regular normal element.  Since $A/(x)\cong k\<y, z\>/(yz-\a zy)$ is a $2$-dimensional quantum polynomial algebra, $\Projn A/(x)\cong \PP^1$ has not fat point.  Since $A[x^{-1}]_0\cong k\<u, v\>/(uv-vu-\a)$ where $u=yx^{-1}, v=zx^{-1}$ is isomorphic to the $1$st Weyl algebra, $\Specn A[x^{-1}]_0$ has no (fat) point by Example \ref{ex.wy}.  By Theorem \ref{thm.dec},
$\Projn A$ has no fat point. Since $E\neq \PP^2$, $||\s||=\infty=|\s^3|$ by Lemma \ref{lem.general}.  

\medskip

\noindent
\underline{Type WL}\quad
Since $A=k\<x, y, z\>/(yz-zy-(1/3)y^2, zx-xz-(1/3)(yx+xy), xy-yx)$, we see that $y\in A_1$ is a regular normal element.  Since $A/(y)\cong k[x, z]$ is a $2$-dimensional (quantum) polynomial algebra, $\Projn A/(y)= \PP^1$ has not fat point.  Since $A[y^{-1}]_0\cong k\<u, v\>/(vu-uv-u)$ where $u=xy^{-1}, v=zy^{-1}$ is isomorphic to the enveloping algebra of a $2$-dimensional non-abelian Lie algebra,
$\Specn A[y^{-1}]_0$ has no fat point by Example \ref{ex.Lie}.  By Theorem \ref{thm.dec}, 
$\Projn A$ has no fat point. Since $E\neq \PP^2$, $||\s||=\infty=|\s^3|$ by Lemma \ref{lem.general}.   

\medskip

\noindent
\underline{Type EC}\quad\noindent
Since $\s_p^i=\s_{ip}\in \Aut_k (\PP^2, E)$ if and only if $ip\in E[3]$ if and only if $3ip=o$, $||\s_p||=|3p|=|\s^3_p|$. 

\medskip 

Next, we will show the equivalences (1) $\Leftrightarrow$ (2) $\Leftrightarrow$ (3) $\Leftrightarrow$ (4).  Since $||\s||=|\s^3|$ for every type, (1) $\Leftrightarrow$ (2).  By Theorem \ref{thm_{ATV2}}, (1) $\Leftrightarrow $ (3).  By definition, (3) $\Rightarrow $ (4), so it is enough to show that (4) $\Rightarrow $ (2).  Indeed, if $\mathsf{Proj}_{{\rm nc}}A$ is finite over its center, 
then
there exists a
$3$-dimensional quantum polynomial algebra 
$A'=\mathcal{A}(E',\s')$ 
which is finite over its center such that 
$\Projn A\cong \Projn A'$ by Definition \ref{def_{Z(proj)}}, 
so $\|\sigma\|=\|\s'\| \leq$
$|\s'|<\infty$ by Lemma \ref{lem_Mo2} and 
Theorem \ref{thm_{ATV2}}. 
\end{proof} 

\subsection{General case} 

\begin{defin}[\textnormal{\cite[Definition 3.2]{MU1}}]
 For a $d$-dimensional geometric quantum polynomial algebra $A=\cA(E, \s)$ with the Nakayama automorphism  $\nu\in \Aut A$,
we define a new graded algebra $\overline A:=\cA(E, \nu^*\s^d)$ satisfying (G2).  
\end{defin} 


\begin{lem}[\textnormal{\cite[Theorem 3.5]{MU1}}]
\label{lem.ovl}
Let $A$ and $A'$ be geometric quantum polynomial algebras.  If $\grmod A\cong \grmod A'$, then $\overline A\cong \overline {A'}$. 
\end{lem}

\begin{rem}
If $A$ and $A'$ are both $3$-dimensional quantum polynomial algebras of the same Type P, S$_1$, S'$_1$, T$_1$, T'$_1$, then the converse of the above lemma was proved in \cite [Theorem 3.6]{MU1}. 
\end{rem} 

\begin{thm} \label{q.nu}  
If $A=\cA(E, \s)$ is a $3$-dimensional quantum polynomial algebra with the Nakayama automorphism $\nu\in \Aut A$, then $||\s||=|\nu^*\s^3|$, so the following are equivalent:
\begin{enumerate}[{\rm (1)}]
\item{} $|\nu^*\s^3|<\infty$. 
\item{} $||\s||<\infty$. 
\item{} $\Projn A$ is finite over its center. 
\end{enumerate}  
Moreover, if $A$ is of Type T, T', CC, TL, WL, then $A$ is never finite over its center. 
\end{thm} 

\begin{proof} For every $3$-dimensional quantum polynomial algebra $A=\cA(E, \s)$, there exists a $3$-dimensional Calabi-Yau quantum polynomial algebra $A'=\cA(E', \s')$ such that $\grmod A\cong \grmod A'$ by Theorem \ref{Main2}. 
Since 
the Nakayama automorphism of $A'$ is the identity, $\cA(E, \nu^*\s^3)=\overline A\cong \overline {A'}=\cA(E', {\s'}^3)$ by Lemma \ref{lem.ovl}, so
$$||\s||=||\s'||=|{\s'}^3|=|\nu^*\s^3|$$ 
by Lemma \ref{lem_Mo2}
and Theorem \ref{q.main}.  Since $\Projn A$ is finite over its center if and only if $\Projn A'$ is finite over its center if and only if $||\s'||<\infty$ by Theorem \ref{q.main}, we have the equivalences (1) $\Leftrightarrow$ (2) $\Leftrightarrow$ (3).

If $A$ is a $3$-dimensional quantum polynomial algebra of Type T, T', CC, TL, WL, then $A'$ is of the same type by Lemma \ref{lem_Mo2}, so $||\s||=||\s'||=\infty$ by the proof of Theorem \ref{q.main}. It follows that $|\s|=\infty$,  
so $A$ is not finite over its center by Theorem \ref{thm_{ATV2}}.    
\end{proof} 


\section{An application to Beilinson algebras} 

We finally apply our results to representation theory of finite dimensional algebras.

\begin{defin}[\textnormal{\cite[Definition 2.7]{HIO}}]
Let $R$ be a finite dimensional algebra of ${\rm gldim} R=d<\infty$. We define an autoequivalence $\nu_d\in \Aut D^b(\mod R)$ by $\nu_d(M):=M\otimes _R^{\rm L}DR[-d]$ where $D^b(\mod R)$ is the bounded derived category of $\mod R$ and $DR:=\Hom_k(R, k)$.   We say that $R$ is {\it $d$-representation infinite} if $\nu_d^{-i}(R)\in \mod R$ for all $i\in \NN$.  In this case, we say that a module $M\in \mod R$ is {\it $d$-regular} if $\nu_d^i(M)\in \mod R$ for all $i\in \ZZ$.  
\end{defin}   

By \cite{Mi}, a $1$-representation infinite algebra is exactly the same as a finite dimensional hereditary algebra of infinite representation type. 
For representation theory of such an algebra, regular modules play an essential role.  

For a $d$-dimensional quantum polynomial algebra $A$, we define {\it the Beilinson algebra of $A$} by 
$$\nabla A:=\begin{pmatrix} A_0 & A_1 & \cdots & A_{d-1} \\
0 & A_0 & \cdots & A_{d-2} \\
\vdots & \ddots & \vdots & \vdots \\
0 & 0 & \cdots & A_0 \end{pmatrix}.$$
The Beilinson algebra is a typical example of $(d-1)$-representation infinite algebra by \cite[Theorem 4.12]{MM}. To investigate representation theory of such an algebra, it is important to classify simple $(d-1)$-regular modules.

\begin{cor} \label{cor.main} 
Let $A=\cA(E, \s)$ be a $3$-dimensional quantum polynomial algebra with the Nakayama automorphism $\nu\in \Aut A$.  Then the following are equivalent:
\begin{enumerate}[{\rm (1)}]
\item{} $|\nu^*\s^3|=1$ or $\infty$. 
\item{} $\Projn A$ has no fat point.
\item{} The isomorphism classes of simple $2$-regular modules over $\nabla A$ are parameterized by the set of closed points of $E\subset \PP^2$.  
\end{enumerate}
In particular, if $A$ is of Type P, T, T', CC, TL, WL, then $A$ satisfies all of the above conditions. 
\end{cor} 

\begin{proof} 
(1) $\Leftrightarrow$ (2): This follow from Theorem \ref{q.nu} and Lemma \ref{lem.general}. 


(2) $\Leftrightarrow$ (3): By \cite[Theorem 3.6]{Mo2}, 
isomorphism classes of simple 2-regular modules over $\nabla A$ are parameterized by the set of points of $\Projn A$.  On the other hand, it is well-known that the points of $\Projn A$ which are not fat (called \textit{ordinary points} in \cite{Mo2}) are parameterized by the set of closed points of $E$ (see \cite[Proposition 4.4]{Mo2}), hence the result holds. 

\end{proof} 

\begin{rem} We have the following characterization of Type P, T, T', CC, TL, WL.  Let $A=\cA(E, \s)$ be a $3$-dimensional quantum polynomial algebra.  Write $E=\cV(f)\subset \PP^2$ where $f\in k[x, y, z]_3$.  Recall that {\it the Hessian of $f$} is defined by $H(f):=\det \begin{pmatrix} f_{xx} & f_{xy} & f_{xz} \\ f_{yx} & f_{yy}  & f_{yz} \\ f_{zx} & f_{zy} & f_{zz} \end{pmatrix}\in k[x, y, z]_3$. Then $A$ is of Type P, T, T', CC, TL, WL if and only if $H^2(f):=H(H(f))=0$. 
\end{rem} 

\begin{rem} 
If $A$ is a 2-dimensional quantum polynomial algebra, 
then $\nabla A\cong \begin{pmatrix} k & k^2 \\ 0 & k \end{pmatrix}\cong 
k(\xymatrix{
\bullet \ar@<0.5ex>[r] 
\ar@<-0.2ex>[r]
& \bullet
})$, 
so $\nabla A$ is a finite dimensional hereditary algebra of tame representation type.  It is known that the isomorphism classes of simple regular modules over $\nabla A$ are parameterized by $\PP^1$ (cf. \cite [Theorem 3.19]{Mo2}).  For a 3-dimensional quantum polynomial algebra $A$, we expect that the following are equivalent: 
\begin{enumerate}
\item{} $\Projn A$ is finite over its center. 
\item{} $\nabla A$ is 2-representation tame in the sense of \cite{HIO}. 
\item{} The isomorphism classes of simple $2$-regular modules over $\nabla A$ are parameterized by $\PP^2$.
\end{enumerate}
These equivalences are shown for Type S in \cite [Theorem 4.17, Theorem 4.21]{Mo2}. 
\end{rem}

\proof[Acknowledgements]
The first author was supported by 
Grants-in-Aid for Young Scientific Research 18K13397 
Japan Society for the Promotion of Science. 
The second author was supported by 
Grants-in-Aid for Scientific Research (C) 20K03510 
Japan Society for the Promotion of Science 
and 
Grants-in-Aid for Scientific Research (B) 16H03923 
Japan Society for the Promotion of Science. 


\begin{thebibliography}{HD}  

\bibitem[AOU]{AOU}
T. ~Abdelgadir, S. ~Okawa and K.~Ueda, 
Compact moduli of noncommutative projective planes, 
preprint (arXiv:1411.7770).

 \bibitem[A]{A}
  M.~Artin, 
 Geometry of quantum planes, 
 Azumaya algebras, actions, and modules (Bloomington, IN, 1990), 
 1--15, \textit{Contemp. Math.}, \textbf{124}, Amer. Math. Soc., Providence, RI, 1992. 

%
 \bibitem[AZ]{AZ} 
  M.~Artin and J.~J.~Zhang, 
  Noncommutative projective schemes, 
  \textit{Adv. Math.} \textbf{109} (1994), no. 2, 228--287.
  
  \bibitem[ATV1]{ATV1}
   M.~Artin, J.~Tate and M.~Van~den~Bergh, 
   Some algebras associated to automorphisms of elliptic curves, 
   The Grothendieck Festschrift, vol. 1, 
   \textit{Progress in Mathematics} vol. 86 (Birkh\"auser, Basel, 1990) 33--85. 

  \bibitem[ATV2]{ATV2}
 \bysame, 
   Module over regular algebras of dimension $3$, 
   \textit{Invent. Math.} \textbf{106} (1991), no. 2, 335--388. 

  
%
%
%
%
%
%
%
%
%

\bibitem[HIO]{HIO}
M.~Herschend, O.~Iyama and S.~Oppermann, 
$n$-representation infinite algebras, 
\textit{Adv. Math.} \textbf{252} (2014), 292--342. 

 \bibitem[IM1]{IM1}
  A.~Itaba and M.~Matsuno, 
  Defining relations of $3$-dimensional quadratic AS-regular algebras, 
  \textit{Math. J. Okayama Univ.} \textbf{63} (2021), 61--86. 
  
  \bibitem[IM2]{IM2}
  \bysame, 
  AS-regularity of geometric algebras of plane cubic curves, 
  submitted (arXiv:1905.02502). 
  
%
   \bibitem[Ma]{Ma}
   M.~Matsuno, 
   A complete classification of $3$-dimensional quadratic AS-regular algebras of Type EC, to appear in \textit{Canad. Math. Bull.}, 
  (arXiv:1912.05167). 
  

\bibitem[Mi]{Mi}
H.~Minamoto,
Ampleness of two-sided tilting complexes,
\textit{Int. Math. Res. Not.} (2012),  no. 1, 67--101. 

\bibitem[MM]{MM}
H.~Minamoto and I.~ Mori,
The structure of AS-Gorenstein algebras,
\textit{Adv. Math.} \textbf{226} (2011), no. 5, 4061--4095. 

%
%
  
   \bibitem[Mo1]{Mo1} 
   I.~Mori,
   Non commutative projective schemes and point schemes, 
   Algebras, Rings and Their Representations, 
   World Sci., Hackensack, N.J., (2006), 215--239. 
   
   \bibitem[Mo2]{Mo2}
\bysame, 
   Regular modules over $2$-dimensional quantum Beilinson algebras of Type S, 
   \textit{Math. Z.} \textbf{279} (2015), no. 3--4, 1143--1174. 
%
%

   \bibitem[MU]{MU1} 
   I.~Mori and K.~Ueyama, 
   Graded Morita equivalences for geometric AS-regular algebras, 
   \textit{Glasg. Math. J.} \textbf{55} (2013), no. 2, 241--257.
  
%
%
%
%
%
\bibitem[S]{S}
S. P.~Smith, 
Noncommutative algebraic geometry, 
lecture notes, 
University of Washington, (1999). 

 

 

  \bibitem[U]{U}
K.~Ueyama, 
Graded Morita equivalences for generic Artin-Schelter regular algebras, 
\textit{Kyoto J. Math.} \textbf{51} (2011), no. 2, 485-501. 
  
   
%
%
%
%

   
\end{thebibliography}
\end{document}